\documentclass[leqno,12pt]{amsart}

\usepackage{graphicx,color}
\usepackage{latexsym}
\usepackage{graphicx} 
\usepackage{amsthm,amsmath, amssymb,amsfonts,esint}
\usepackage{layout,textcomp}

\theoremstyle{plain}

\theoremstyle{plain}
\newtheorem{theorem}{Theorem}
\newtheorem*{Thm}{Theorem}
\newtheorem{corollary}[theorem]{Corollary}
\newtheorem{lemma}[theorem]{Lemma}
\newtheorem{proposition}[theorem]{Proposition}

\theoremstyle{definition}

\theoremstyle{remark}

\DeclareMathOperator{\Hess}{Hess}

\begin{document}

\title[A gap theorem]{A gap theorem for free boundary minimal surfaces in the three-ball}

\author[Lucas Ambrozio]{Lucas Ambrozio}
\address{Departament of Mathematics, Imperial College London, South Kensington Campus, London, UK.}
\email{l.ambrozio@imperial.ac.uk}

\author[Ivaldo Nunes]{Ivaldo Nunes}
\address{Departamento de Matem\'atica, Universidade Federal do Maranh\~ao, Campus do Bacanga, S\~ao Lu\'is, MA, Brasil.}
\address{Departament of Mathematics, Imperial College London, South Kensington Campus, London, UK. (\textit{Current address})}
\email{ivaldo.nunes@ufma.br}

\date{}
\begin{abstract}
We show that, among free boundary minimal surfaces in the unit ball in the three-dimensional Euclidean space, the flat equatorial disk and the critical catenoid are characterised by a pinching condition on the length of their second fundamental form.
\end{abstract}
\maketitle
\section{Introduction}

\indent Let $B^3$ denote the unit ball in $\mathbb{R}^3$, centred at the origin. The critical points of the area functional on compact surfaces in $B^3$ whose boundaries are free to move in $\partial B$ are minimal surfaces that intersect $\partial B$ orthogonally. \\
\indent The simplest examples of free boundary minimal surfaces in $B^3$ are the flat equatorial disk and the critical catenoid, which will be described in details below. Fraser and Schoen \cite{FraSch1} constructed embedded examples of genus zero with any number of boundary components by finding metrics on surfaces of this topological type that maximize the first Steklov eigenvalue with fixed boundary length. The gluing construction of Pacard, Folha and Zolotareva \cite{PacFolZol} yielded free boundary minimal surfaces in $B^3$ of genus zero or one, and with any number of boundary components greater than a large constant. The existence of examples with three boundary components and arbitrarily large genus has been announced by M. Li and Kapouleas. Although expected, it is not clear yet whether a free boundary minimal surface in $B^3$ can be of arbitrary genus and have any number of boundary components. \\
\indent There are classifications theorems that single out the flat equatorial disk and the critical catenoid from among free boundary minimal surfaces in $B^3$. Some of them will be reviewed in Section 2. In this work, we sought for a new characterisation of the flat equatorial disk and the critical catenoid in terms of the length of their second fundamental forms. It turns out that a pinching condition on this quantity, that takes into account the support function of the surface, is sufficient. The result we obtained is the following:

\begin{Thm}\label{maintheorem}
Let $\Sigma^2$ be a compact free boundary minimal surface in $B^{3}$. Assume that for all points $x$ in $\Sigma$,
\begin{equation} \label{pinch}
 |A|^2(x)\langle x,N(x)\rangle^2\leqslant 2,
\end{equation}
\noindent where $N(x)$ denotes a unit normal vector at the point $x\in \Sigma$ and $A$ denotes the second fundamental form of $\Sigma$. Then
\begin{itemize}
\item[$i)$] either $|A|^2\langle x,N\rangle^2\equiv 0$ and $\Sigma$ is a flat equatorial disk;
\item[$ii)$] or $|A|^2(p)\langle p,N(p)\rangle^2=2$ at some point $p\in\Sigma$ and $\Sigma$ is a critical catenoid.
\end{itemize}
\end{Thm}
\indent The result above bears some resemblance with the following theorem, which answered a question raised by the work of Simons \cite{Sim}:
\begin{Thm}[Chern-do Carmo-Kobayashi \cite{ChedoCKob}, Lawson \cite{Law}]
 Let $\Sigma^n$ be a closed minimal hypersurface in the unit sphere $S^{n+1}$. Assume that 
 \begin{equation*}
  |A|^2\leqslant n.
 \end{equation*}
 Then 
 \begin{itemize}
 \item[$i)$] either $|A|^2\equiv 0$ and $\Sigma^n$ is an equator;
 \item[$ii)$] or $|A|^2\equiv n$ and $\Sigma^n$ is a Clifford hypersurface (product of two round spheres of appropriate radii and dimensions).
 \end{itemize}
\end{Thm}  
\indent However, the proof of our theorem is quite different and some steps are true only in low dimensions. It would be interesting to know if our result can be generalised to higher ambient dimension and submanifold co-dimensions, providing a full analogy between the results of \cite{ChedoCKob} and this free boundary setting.

\section{The flat equatorial disk and the critical catenoid}

\indent The flat equatorial disks are defined by the intersection of the unit ball $B^3$ with a plane containing the origin. \\
\indent The flat equatorial disks are the only totally geodesic free boundary minimal surfaces in $B^3$. Nitsche \cite{Nit} has shown that they are in fact the only immersed free boundary minimal surfaces in $B^3$ that are homeomorphic to a disk. An interesting generalisation of his result to higher codimensions was obtained by Fraser and Schoen in \cite{FraSch2}. In a different direction, it is known that the flat equatorial disk has the least possible area among free boundary minimal surfaces in $B^3$ (see for example \cite{RosVer}, Proposition 3, or \cite{FraSch3}, Theorem 5.4). This was generalised by Brendle \cite{Bre}. Finally, flat equatorial disks are the only free boundary minimal surfaces with Morse index equal to one (a more general statement can be found in \cite{FraSch1}, Theorem 3.1). \\
\indent Critical catenoids, on the other hand, are the only non-flat free boundary minimal surfaces in $B^3$ that are invariant under rotations around a given axis. They can be analytically defined as the image of the conformal maps
\begin{multline*}
 X : (t,\theta) \in [-t_0,t_0]\times S^1 \\
 \mapsto a_0\cosh(t)\cos(\theta)e_1+a_0\cosh(t)\sin(\theta)e_2 + a_0t e_3 \in \mathbb{R}^3
\end{multline*}
\noindent where $\{e_1,e_2,e_3\}$ is some orthonormal basis of $\mathbb{R}^3$. The constant $t_0$ is the unique positive solution to the equation $t\sinh(t)=\cosh(t)$, while $a_0=(t_0\cosh(t_0))^{-1}$. These constants are chosen in such way that the map $X$ is conformal and its image is a piece of a catenoid contained in $B^3$, meeting $\partial B$ orthogonally, and whose axis of symmetry is the line generated by the vector $e_3$. \\
\indent It has been conjectured that the critical catenoid is the only embedded free boundary minimal annulus in $B^3$ (for example, in \cite{FraLi}, Conjecture 1.1). McGrath \cite{Mcg} showed that this conjecture is true under the additional assumption that the surface is symmetric with respect to reflections through three mutually orthogonal planes. Fraser and Schoen \cite{FraSch1} proved that the critical catenoid is the only free boundary minimal annulus that is immersed in $B^3$ by its first Steklov eigenfunctions. It would be also very interesting to classify the critical catenoid only by its Morse index. \\ 
\indent Regarding our main result, we observe that, in the parametrization $X$ given above, the second fundamental form and the support function are given by
\begin{equation*}
 |A|^2 = \frac{2}{a_0^2\cosh^4(t)} \quad \text{and} \quad \langle x,N\rangle^2 = a_0^2\left(1 - \frac{t\sinh(t)}{\cosh(t)}\right)^2.
\end{equation*}
\indent In particular, since $|t|\leq t_0$,
\begin{equation*}
 |A|^2\langle x,N\rangle^2 = \frac{2}{\cosh^6(t)}\left(\cosh(t) - t\sinh(t)\right)^2 \leqslant 2.
\end{equation*}
\noindent Notice that the maximum value $2$ is attained at $t=0$, that is, on the circle defined by the intersection of the given critical catenoid and the plane through the origin that is orthogonal to its axis of symmetry.

\section{Proof of the Theorem}

\indent In this section, we explain the proof of our main result. The key observation is to relate the pinching condition (\ref{pinch}) to a condition on the Hessian of the distance function of points on $\Sigma$ to the origin.
\begin{lemma}\label{distancefunction}
Let $\Sigma^2$ be a free boundary minimal surface in $B^3$. Let $f$ be the function defined by 
\begin{equation*}
f(x)=\frac{|x|^2}{2}, \quad x\in\Sigma.
\end{equation*}
Then,
\begin{itemize}
\item[$i)$] $\nabla^\Sigma f(x)=x$ for all $x\in\partial\Sigma$.
\item[$ii)$] For each $x\in\Sigma$, the eigenvalues of $\Hess_\Sigma f(x)$ are given by 
$$
1-\frac{|A|(x)}{\sqrt{2}}\langle x,N(x)\rangle \ \ \mbox{and} \ \ 1 +\frac{|A|(x)}{\sqrt{2}}\langle x,N(x)\rangle .
$$
\end{itemize}
\end{lemma}
\begin{proof}
In order to prove $i)$, we first note that $\nabla^\Sigma f(x)=x^T$ for all $x\in\Sigma$, where $(\cdot)^T$ stands for the orthogonal projection onto $T_x\Sigma$. Since $\Sigma$ meets $\partial B$ orthogonally along $\partial\Sigma$, $x\in T_x\Sigma$ for all $x\in\partial\Sigma$ and the assertion follows. \\
\indent Given $x\in\Sigma$, let $N$ be a local unit normal vector field in a neighbourhood of $x$. For all $X,Y\in T_x\Sigma$, we have
\begin{align*}
\Hess_\Sigma f(x)(X,Y)&=XY(f)-(\nabla_XY)(f)\\
&=X\langle x,Y\rangle-\langle x,\nabla_XY\rangle\\
&=\langle X,Y \rangle+\langle x, D_XY\rangle-\langle x,\nabla_XY\rangle\\
&=\langle X,Y\rangle-\langle A(X),Y\rangle\langle x,N(x)\rangle\\
&=\langle X-A(X)\langle x, N(x)\rangle,  Y\rangle.
\end{align*} 
\indent Therefore the eigenvalues of $\Hess_\Sigma f(x)$ are given by 
\begin{equation*}
 1-k_1\langle x,N(x)\rangle \quad \text{and} \quad 1-k_2\langle x,N(x)\rangle,
\end{equation*}
\noindent where $k_1\leqslant k_2$ are the principal curvatures of $\Sigma$ at $x$. Since $\Sigma$ is minimal, $k_1=-|A|(x)/\sqrt{2}$ and $k_2=|A|(x)/\sqrt{2}$. This finishes the proof of $ii)$.
\end{proof}

\indent Under the assumptions of our theorem, we then conclude that the function $f$ defined in Lemma  \ref{distancefunction} must be convex. This imposes very strong restrictions on the geometry of the surface $\Sigma$.

\begin{proposition}\label{structureofC}
Let $\Sigma^2$ be a compact free boundary minimal surface in $B^3$. Define
\begin{equation*}
 \mathcal{C} = \{p\in \Sigma;\, |p|=\min_{x\in \Sigma} |x|\}.
\end{equation*} 
\indent If $|A|^2\langle x,N\rangle^2\leq 2$ on $\Sigma$, then
\begin{itemize}
\item[$i)$] either $\mathcal{C}$ contains a single point $p\in\Sigma\setminus\partial\Sigma$, in which case $\Sigma$ must be a flat equatorial disk;
\item[$ii)$] or $\mathcal{C}$ is a simple closed geodesic in $\Sigma\setminus\partial\Sigma$ and $\Sigma$ is homeomorphic to an annulus.
\end{itemize}
\end{proposition}

\begin{proof}
\indent Let $f:\Sigma\to\mathbb{R}$ be defined as in Lemma \ref{distancefunction}. We claim that $\mathcal{C}$ is a totally convex subset of $\Sigma$. Recall that a subset $A$ of a Riemannian manifold $(M^n,g)$ is totally convex when any geodesic in $M^n$ joining two points in $A$ actually lies in $A$ (see, for example, \cite{CheGro}). In fact, given $p,q\in\mathcal{C}$, let $\gamma:[0,1]\to\Sigma$ be a geodesic such that $\gamma(0)=p$ and $\gamma(1)=q$. By item (2) in Lemma \ref{distancefunction}, the geometric condition $|A|^2\langle x,N\rangle^2\leq 2$ on $\Sigma$ is equivalent to $\Hess_\Sigma f\geq 0$ on $\Sigma$. Hence, $(f\circ \gamma)^{\prime\prime}(t)\geq 0$ for all $t\in[0,1]$, that is, $f\circ \gamma$ is convex on $[0,1]$. By the definition of $\mathcal{C}$, as $f\circ \gamma$ attains its minimum at $t=0$ and $t=1$ we conclude that $(f\circ \gamma)(t)\equiv \min_\Sigma f$. Therefore $\gamma([0,1])\subset\mathcal{C}$ and the claim follows. \\ 
\indent Since $\Sigma$ meets $\partial B$ orthogonally, the geodesic curvature $k_g$ of $\partial \Sigma$ in $\Sigma$, computed with respect to the unit outward pointing conormal, satisfies $k_g\equiv 1$. In particular, $\partial\Sigma$ is strictly convex in $\Sigma$. This implies that for all $p,q\in\Sigma$ there exists a minimising geodesic in $\Sigma$ joining $p$ to $q$. Thus, the totally convex set $\mathcal{C}$ is connected. Also, $\mathcal{C}$ is contained in the interior of $\Sigma$, because the gradient of $f$ is non-zero on $\partial \Sigma$ and points outwards, by Lemma \ref{distancefunction}, item $i)$. \\ 
\indent Now, let us suppose that $\mathcal{C}$ contains only one point $p\in\Sigma\setminus \partial \Sigma$. Given $[\alpha]\in\pi_1(\Sigma,p)$, let us assume that $[\alpha]$ is a nontrivial homotopy class. Since $\partial\Sigma$ is strictly convex, we can find a geodesic loop $\gamma:[0,1]\to\Sigma$, $\gamma(0)=\gamma(1)=p$, such that $\gamma\in[\alpha]$. Since $\mathcal{C}$ is totally convex, $\gamma([0,1])\subset\mathcal{C}$. But $\mathcal{C}=\{p\}$ and $[\alpha]$ is non-trivial, which is a contradiction. Therefore $\Sigma$ is homeomorphic to a disk. By Nitsche's Theorem \cite{Nit}, $\Sigma$ is a flat equatorial disk. \\
\indent Finally, suppose that $\mathcal{C}$  does not consist of a single point. By Nitsche's result, $\Sigma$ cannot be homeomorphic to a disk. Let $p\in\mathcal{C}$ and let $[\alpha]\in\pi_1(\Sigma,p)$ be a non-trivial homotopy class. As above, since $\partial\Sigma$ is strictly convex, we can find a geodesic loop $\gamma:[0,1]\to\Sigma$, $\gamma(0)=\gamma(1)$, such that $\gamma\in[\alpha]$. We claim that $\gamma^\prime(0)=\gamma^\prime(1)$, $\gamma([0,1])$ is a simple curve and $\mathcal{C}=\gamma([0,1])$. In fact, if $\gamma^\prime(0)\neq\gamma^\prime(1)$, then, because $\mathcal{C}$ is totally convex, it is possible to join points on $\gamma$ near the break at $p$ by minimising geodesics and find an open set $U\subset\mathcal{C}$, which is a contradiction as $\Sigma$ is minimal. By similar reasoning, $\gamma$ has to be simple and $\mathcal{C}=\gamma([0,1])$, so that $\mathcal{C}$ is a simple closed geodesic. Since the above argument also shows that any geodesic loop based at $p$ must be contained in $\mathcal{C}$, we conclude that $\pi_1(\Sigma,p)=\mathbb{Z}$, that is, $\Sigma$ is homeomorphic to an annulus.
\end{proof}

\indent The following corollary is a direct consequence of the above proof.
\begin{corollary}\label{corollary1}
Let $\Sigma^2$ be a compact free boudary minimal surface in $B^3$. If $|A|^2\langle x,N\rangle^2<2$ on $\Sigma$, then $\Sigma$ is a flat equatorial disk.
\end{corollary}

\begin{proof}
Since $|A|^2\langle x,N\rangle^2<2$, by item $ii)$ of Lemma \ref{distancefunction} it follows that the function $f$ defined in Proposition \eqref{distancefunction} is strictly convex, that is, $\Hess_\Sigma f>0$ on $\Sigma$. By item $i)$ of the same Lemma, $f$ attains its mininum value on the interior of $\Sigma$. By the strict convexity of $f$, the set of minima contains a single point. Thus, Proposition \ref{structureofC} implies that $\Sigma$ is a flat equatorial disk.
\end{proof}
\indent In order the finish the proof of our main result, it remains only to analyse the situation where the function $|A|^2\langle x,N\rangle^2$ attains the maximum value $2$ at some point on $\Sigma$.

\begin{proposition} Let $\Sigma^2$ be a compact free boundary minimal surface in $B^3$. If $|A|^2\langle x,N\rangle^2\leq 2$ on $\Sigma$ and $|A|^2(p)\langle p,N(p)\rangle^2=2$ at some point $p\in\Sigma$, then $\Sigma$ is a critical catenoid.
\end{proposition}

\begin{proof}
Let $f:\Sigma\to\mathbb{R}$ and $\mathcal{C}$ be defined as in Lemma \ref{distancefunction} and Proposition \ref{structureofC}, respectively. Since $|A|^2(p)\langle p,N(p)\rangle^2=2$ at some point $p\in\Sigma$, the surface $\Sigma$ cannot be homeomorphic to a disk, for in this case it would be totally geodesic as a consequence of Nitsche's Theorem \cite{Nit}. Thus, by Proposition \ref{structureofC}, $\Sigma$ is homeomorphic to an annulus and $\mathcal{C}$ is a simple closed geodesic, say $\gamma:[0,\ell]\to\Sigma$, parametrized by arc length, in $\Sigma\setminus\partial\Sigma$. In particular, $\inf_\Sigma f>0$. \\
\indent 
Let $R>0$ be given by $R^2/2=\inf_\Sigma f$ and let $S_R\subset\mathbb{R}^3$ be the sphere of radius $R$ centred at the origin. We claim that $\gamma$ is a great circle in $S_R$. In fact, by definition of $R$, we have that $\Sigma\subset\{x\in\mathbb{R}^3:|x|\geqslant R\}$ and $\Sigma\cap S_R=\gamma([0,\ell])$. This implies that $T_{\gamma(s)}\Sigma=T_{\gamma(s)}S_R$ for all $s\in[0,\ell]$. Since $\gamma$ is a geodesic of $\Sigma$, we then conclude that $\gamma$ is also a geodesic of $S_R$, that is, $\gamma$ is a great circle of $S_R$. \\
\indent If $e_3$ denote a unit vector in $\mathbb{R}^3$ that is orthogonal to the plane containing the great circle $\gamma$, then $\{\gamma'(s),e_3\}$ is an orthonormal basis of $T_{\gamma(s)}\Sigma$ for all $s\in[0,\ell]$. \\
\indent Let $u(x)=\langle x\wedge N(x),e_3\rangle$, $x\in\Sigma$, where $\wedge$ denotes the cross product in $\mathbb{R}^3$. The function $u$ is the infinitesimal normal speed of the variation of $\Sigma$ by rotations around the $e_3-$axis. As such, $u$ is a Jacobi function of $\Sigma$, that is, 
\begin{equation*}
\Delta_\Sigma u+|A|^2u=0\ \ \mbox{on} \ \ \Sigma, \ \ \mbox{and} \ \ \frac{\partial u}{\partial \nu}=u \ \ \mbox{on} \ \ \partial\Sigma.
\end{equation*}
\noindent Moreover, $u$ vanishes identically if and only if $\Sigma$ is a surface of revolution around the $e_3-$axis. \\
\indent Clearly, $\gamma([0,\ell])$ is contained in $u^{-1}(0)$, which immediately implies that $d u(\gamma(s))\gamma^\prime(s)=0$ for all $s\in[0,\ell]$. On the other hand, as $\gamma^\prime(s)$ is a principal direction of $\Sigma$ at $\gamma(s)$, so it is the orthogonal direction $e_3\in T_{\gamma(s)}\Sigma$. Thus, $dN_{\gamma(s)}e_3$ is parallel to $e_3$ and
\begin{equation*}
du(\gamma(s))e_3=\langle e_3\wedge N(\gamma(s)), e_3\rangle + \langle \gamma(s)\wedge dN_{\gamma(s)}e_3,e_3\rangle=0.
\end{equation*}
\indent The above argument has shown that every point in the circle $\mathcal{C}$, which is contained in $u^{-1}(0)$, is a critical point of $u$. But this implies that $u$ vanishes identically on $\Sigma$. In fact, should this not be true, the nodal set $u^{-1}(0)$ of the Jacobi function $u$ would contain only isolated critical points, by a result due to S. Y. Cheng (\cite{Che}, Theorem 2.5). And this would be a contradiction. \\
\indent Thus, $\Sigma$ is an free boundary minimal annulus of revolution around the $e_3-$axis in $B^3$. In other words, $\Sigma$ is a critical catenoid.
\end{proof}

\noindent \textbf{Acknowledgements:} L. A. is supported by the ERC Start Grant PSC and LMCF 278940 and I. N. is supported by CNPq-Brazil. Both authors are grateful to Andr\'e Neves, Celso Viana, Rafael Montezuma Fernando Cod\'a Marques, Ben Sharp and Alessandro Carlotto for their kind interest in this work. \\

\end{document}